\documentclass[a4paper,11pt,oneside]{amsart}
\usepackage[dvipsnames]{xcolor}
\usepackage[T1]{fontenc}
\usepackage[utf8]{inputenc}
\usepackage{geometry}
\usepackage{lmodern} 
\usepackage{microtype}

\title{Trace conjunction inequalities}

\author{Jean Van Schaftingen}
\address{Universit\'e catholique de Louvain\\ 
Institut de Recherche en Math\'ematique et Physique\\
Chemin du Cyclotron 2 bte L7.01.01\\
1348 Louvain-la-Neuve\\
Belgium}
\email{Jean.VanSchaftingen@UCLouvain.be}

\usepackage{enumerate}

\usepackage{amssymb}
\usepackage{amsmath}
\usepackage{amsthm}
\usepackage{esint}

\usepackage{constants}
\usepackage{paralist}
\usepackage{mathrsfs}

\newcommand{\dimdom}{N}

\usepackage[unicode=true, pdflang={en-GB}, pdfusetitle,
allbordercolors=RoyalBlue,
pdfborderstyle={/W .5},
final,
pagebackref]{hyperref}

\usepackage[abbrev,backrefs]{amsrefs}

\usepackage[nameinlink%
]{cleveref}

\renewcommand{\PrintDOI}[1]{%
  \href{https://doi.org/#1}{doi:#1}%
}

\newtheorem{theorem}{Theorem}[section]
\newtheorem{proposition}[theorem]{Proposition}
\newtheorem{lemma}[theorem]{Lemma}

\theoremstyle{definition}

\newtheorem{openproblem}{Open Problem}

\theoremstyle{remark}

\numberwithin{equation}{section}

\dedicatory{In blessed memory of Haïm Brezis\\ and his joy of asking and solving problems}

\usepackage{mathtools}

\DeclarePairedDelimiter{\brk}{(}{)}
\DeclarePairedDelimiter{\sqb}{[}{]}
\DeclarePairedDelimiter{\abs}{\lvert}{\rvert}
\DeclarePairedDelimiter{\norm}{\lVert}{\rVert}

\DeclarePairedDelimiterX{\intvc}[2]{[}{]}{#1,#2}
\DeclarePairedDelimiterX{\intvl}[2]{(}{]}{#1,#2}
\DeclarePairedDelimiterX{\intvr}[2]{[}{)}{#1,#2}
\DeclarePairedDelimiterX{\intvo}[2]{(}{)}{#1,#2}
\DeclarePairedDelimiterX{\setcond}[2]{\{}{\}}{#1 \,\delimsize\vert\, #2}
\newcommand\stSymbol[1][]{%
\nonscript\;#1\vert
\allowbreak
\nonscript\;
\mathopen{}}
\DeclarePairedDelimiterX\set[1]\{\}{%
\renewcommand\st{\stSymbol[\delimsize]}
#1
}
\providecommand{\st}{\stSymbol}

\newcommand{\defeq}{\coloneqq}

\newcommand{\Rset}{\mathbb{R}}
\newcommand{\Nset}{\mathbb{N}}

\newcommand{\Bset}{\mathbb{B}}

\newcommand{\dif}{\,\mathrm{d}}

\DeclareMathOperator*{\trace}{tr}
\newcommand{\Deriv}{\mathrm{D}}

\DeclareMathOperator{\supp}{supp}

\newcommand{\sobolev}{\smash{\dot{W}}}

\BibSpec{book}{%
    +{}  {\PrintPrimary}                {transition}
    +{,} { \textit}                     {title}
    +{.} { }                            {part}
    +{:} { \textit}                     {subtitle}
    +{,} { \PrintEdition}               {edition}
    +{}  { \PrintEditorsB}              {editor}
    +{,} { \PrintTranslatorsC}          {translator}
    +{,} { \PrintContributions}         {contribution}
    +{,} { }                            {series}
    +{,} { \voltext}                    {volume}
    +{,} { }                            {publisher}
    +{,} { }                            {organization}
    +{,} { }                            {address}
    +{,} { \PrintDateB}                 {date}
    +{,} { }                            {status}
    +{,} { \PrintDOI}                   {doi}
    +{}  { \parenthesize}               {language}
    +{}  { \PrintTranslation}           {translation}
    +{;} { \PrintReprint}               {reprint}
    +{.} { }                            {note}
    +{.} {}                             {transition}
    +{}  {\SentenceSpace \PrintReviews} {review}
}

\begin{document}
\begin{abstract}
Trace conjunction integrals are introduced and studied.
They appear in trace conjunction inequalities which unify the Hardy inequality on a halfspace and the classical Gagliardo trace inequality.
At the endpoint they satisfy a Bourgain-Brezis-Mironescu formula for smooth maps, which raises some new open problems.
\end{abstract}

\keywords{Sobolev space; Hardy inequality; trace inequality; fractional Sobolev spaces.}
\subjclass[2020]{46E35}
\thanks{J. Van Schaftingen was supported by the Projet de Recherche T.0229.21 ``Singular Harmonic Maps and Asymptotics of Ginzburg--Landau Relaxations'' of the Fonds de la Recherche Scientifique--FNRS}

\maketitle

\section{Introduction}
The \emph{homogeneous Sobolev space}
\begin{equation}
 \sobolev^{1, p} \brk{\Rset^{\dimdom}_+}
 \defeq 
 \set[\bigg]{u \colon \Rset^{\dimdom}_+ \to \Rset \st
 u \text{ is weakly differentiable and } 
 \int_{\Rset^{\dimdom}_+} \abs{\Deriv u}^p < \infty
 },
\end{equation}
with \(\dimdom \in \Nset \setminus \set{0, 1}\),
\(\Rset^{\dimdom}_+ = \Rset^{\dimdom-1} \times \intvo{0}{\infty}\) and \(p > 1\), consists a priori of equivalence classes of merely measurable functions that should not have well-defined restrictions to Lebesgue null sets.
However the integrability condition on the derivative has long been known to allow the definition of \emph{traces} on lower-dimensional sets such as the boundary \(\partial \Rset^{\dimdom}_+ \simeq \Rset^{\dimdom - 1}\) \cite{Stampacchia_1952} (see also \citelist{\cite{Brezis_2011}*{lem.\ 9.9}\cite{Willem_2022}*{prop.\ 6.2.3}}).

Since the works of Gagliardo \cite{Gagliardo_1957} (see also \citelist{\cite{Aronszajn_1955}\cite{Prodi_1957}\cite{Slobodecki_Babich_1956}} for \(p = 2\)), it has been known that if \(p > 1\), every function \(u \in \sobolev^{1, p} \brk{\smash{\Rset^{\dimdom}_+}}\) has a trace \(v = \smash{\trace_{\partial \Rset^{\dimdom}_+} u}
\in \sobolev^{1-1/p, p}\brk{\smash{\partial \Rset^{\dimdom}_+}}\),
where the \emph{homogeneous fractional Sobolev space} \(\smash{\sobolev^{s, p}} \brk{\Rset^{\ell}}\) is defined for \(\ell \in \Nset \setminus \set{0}\),
\(s \in \intvo{0}{1}\) and \(p \in \intvr{1}{\infty}\) as
\begin{equation}
\label{eq_iecai9naechie0pei5Noohae}
 \sobolev^{s, p} \brk{\Rset^{\ell}}
 \defeq 
 \set[\bigg]{v \colon \Rset^{\ell} \to \Rset \st 
 \smashoperator{\iint_{\Rset^{\ell}\times \Rset^{\ell}}}
 \frac{\abs{v \brk{x} - v \brk{y}}^p}{\abs{x - y}^{\ell + sp}}\dif x \dif y < \infty}.
\end{equation}
Moreover, the trace operator \(\smash{\trace_{\partial \Rset^{\dimdom}_+}} \colon \sobolev^{1, p} \brk{\Rset^{\dimdom}_+}
\to \sobolev^{1-1/p, p}\brk{\partial \Rset^{\dimdom}_+}\) is continuous:
there is some constant \(C \in \intvo{0}{\infty}\) such that
if \(u \in \sobolev^{1, p} \brk{\Rset^{\dimdom}_+}\) and \(v = \smash{\trace_{\partial \Rset^{\dimdom}_+} u}\),
the \emph{Gagliardo integral} which appears in the definition  \eqref{eq_iecai9naechie0pei5Noohae} of \(\sobolev^{1 - 1/p, p} \brk{\partial \Rset^{\dimdom}_+}\) satisfies the trace inequality
\begin{equation}
\label{eq_Epieshu8ohr2iecuewei1eed}
\smashoperator{
 \iint_{\partial \Rset^{\dimdom}_+ \times\partial  \Rset^{\dimdom}_+}}
 \frac{\abs{v \brk{x} - v \brk{y}}^p}{\abs{x - y}^{\dimdom + p - 2}}\dif x \dif y
 \le C \int_{\Rset^{\dimdom}_+}
 \abs{\Deriv u}^p.
\end{equation}
The trace operator \(\smash{\trace_{\partial \Rset^{\dimdom}_+}}\) also has a linear continuous right inverse that provides an extension of boundary data. At the endpoint \(p = 1\), the range of trace turns out to be \(L^1 \brk{\partial \Rset^{\dimdom}_+}\) \cite{Gagliardo_1957} (see also \cite{Mironescu_2015}) without a corresponding linear continuous right inverse \cite{Peetre_1979} (see also \cite{Pelczynski_Wojciechowski_2002}).

The aim of the present work is to introduce the \emph{trace conjunction integral} that connects a function to its trace quantitatively and qualitatively.
Our starting point is the inequality
\begin{equation}
\label{eq_eexuevaeyeaghooPhie6ohwa}
  \int_{\partial \Rset^{\dimdom}_+} 
 \brk[\bigg]{\int_{\Rset^{\dimdom}_+}
 \frac{\abs{v \brk{x}- u \brk{y}}^p}
 {\abs{x - y}^{\dimdom + p - 1}}
 \dif y}
 \dif x
 \le
 C 
 \int_{\Rset^{\dimdom}_+} \abs{\Deriv u }^p,
\end{equation}
for \(u \in \sobolev^{1, p} \brk{\Rset^{\dimdom}_+}\) and \(v = \trace_{\partial \Rset^{\dimdom}} u\) (see \cref{theorem_first_order_trace_conjunction} below).
The trace conjunction integral is the quantity on the left-hand side of \eqref{eq_eexuevaeyeaghooPhie6ohwa}.
It is a kind of mixed double integral which integrates the distance between the boundary values and the interior values.

Trace conjunction integrals have already been known to  appear in Fubini type arguments, where one writes
\begin{equation}
\label{eq_ohghuvei8aeJoz8aipoxeich}
 \smashoperator{\iint_{\Rset^{\dimdom}\times \Rset^{\dimdom}}}
 \frac{\abs{u \brk{x} - u \brk{y}}^p}{\abs{x - y}^{\dimdom + sp}}\dif x \dif y
 = 
 \int_{\Rset^{\dimdom - \ell}}
 \int_{\Rset^{\ell} \times \set{x''}}
 \int_{\Rset^{\dimdom}}
 \frac{\abs{u \brk{x} - u \brk{y}}^p}{\abs{x - y}^{\dimdom + sp}}\dif y \dif x' \dif x'',
\end{equation}
with \(x = \brk{x', x''}\),
before applying a triangle inequality argument to rewrite the two innermost integrals of the right-hand side of \eqref{eq_ohghuvei8aeJoz8aipoxeich} as a double integral on \(\brk{\Rset^{\ell} \times \set{x''}}^2\) (see for example \cite{VanSchaftingen_2024}*{prop.\ 5.9});
a similar argument proves
that the trace conjunction integral on the left-hand side of \eqref{eq_eexuevaeyeaghooPhie6ohwa}
controls the Gagliado energy on the left-hand side of \eqref{eq_Epieshu8ohr2iecuewei1eed}  (\cref{theorem_mixed_to_gagliardo}), so that \eqref{eq_eexuevaeyeaghooPhie6ohwa} implies \eqref{eq_Epieshu8ohr2iecuewei1eed}.
The trace conjunction inequality \eqref{eq_eexuevaeyeaghooPhie6ohwa} turns out to provide an interesting route to the trace inequality \eqref{eq_Epieshu8ohr2iecuewei1eed}, with a proof of \eqref{eq_eexuevaeyeaghooPhie6ohwa} that avoids the need to introduce somehow artificial interior intermediate points in the middle of the proof and treat similarly two symmetric terms.

The trace conjunction inequality \eqref{eq_eexuevaeyeaghooPhie6ohwa} also contains the information that \(v\) is the trace of \(u\), in the sense that if \(v \colon \partial \Rset^{\dimdom}_+\to \Rset\) is any function for which the left-hand side of \eqref{eq_eexuevaeyeaghooPhie6ohwa} is finite, then \(v\) is necessarily the trace of \(u\) (\cref{theorem_trace_characterisation}).

The trace conjunction inequality \eqref{eq_eexuevaeyeaghooPhie6ohwa} is also connected to the classical \emph{Hardy inequality} \citelist{\cite{Hardy_1920}\cite{Hardy_Littlewood_Polya}*{th.\ 327}} (see also \citelist{\cite{Mironescu_2018}\cite{Davies_1998}\cite{Kufner_Maligranda_Persson_2007}})
\begin{equation}
\label{eq_maiHieloib3eijohgiephahc}
 \int_{\Rset^{\dimdom}_+}
 \frac{\abs{v \brk{x'} - u \brk{x}}^p}{x_{\dimdom}^{p}} \dif x
 \le \brk[\bigg]{\frac{p}{p - 1}}^p
 \int_{\Rset^{\dimdom}_+}
 \abs{\Deriv u}^p,
\end{equation}
for \(u \in \sobolev^{1, p} \brk{\Rset^{\dimdom}_+}\), \(v = \smash{\trace_{\partial \Rset^{\dimdom}_+}}u\) and
with \(x = \brk{x', x_{\dimdom}}\).
Indeed, a straightforward argument shows that the left-hand side of the Hardy inequality \eqref{eq_maiHieloib3eijohgiephahc} is controlled by the trace conjunction integral on the left-hand side of \eqref{eq_eexuevaeyeaghooPhie6ohwa},
 so that the trace conjunction inequality \eqref{eq_eexuevaeyeaghooPhie6ohwa} implies the Hardy inequality \eqref{eq_maiHieloib3eijohgiephahc}.
Compared with the Hardy inequality \eqref{eq_maiHieloib3eijohgiephahc}, the trace conjunction inequality \eqref{eq_eexuevaeyeaghooPhie6ohwa} is invariant under suitable change of coordinates by a diffeomorphism, which makes it more appealing in more geometrical contexts.

As the classical trace \eqref{eq_Epieshu8ohr2iecuewei1eed} and Hardy   \eqref{eq_maiHieloib3eijohgiephahc} inequalities could be derived from the trace conjunction inequality \eqref{eq_eexuevaeyeaghooPhie6ohwa},
the latter can be derived from the two former in the sense that the left-hand side of \eqref{eq_eexuevaeyeaghooPhie6ohwa} can be controlled by the left-hand side of \eqref{eq_Epieshu8ohr2iecuewei1eed} and \eqref{eq_maiHieloib3eijohgiephahc} (\cref{theorem_Hardy_Gagliardo_conjunction}).

Besides the first-order Sobolev spaces \(\sobolev^{1, p} \brk{\Rset^{\dimdom}_+}\) considered above, other flavours of Sobolev spaces are known to have traces.
We prove trace conjunction inequalities similar to \eqref{eq_eexuevaeyeaghooPhie6ohwa} for some wider range of weighted Sobolev spaces (\cref{theorem_first_order_trace_conjunction}) and fractional Sobolev spaces (\cref{theorem_fractional_trace_conjunction}).

As further applications, the sharp connection between a function and its trace expressed by the trace conjunction integral could be used to prove in the critical case \(p = \dimdom\) and \(s = 1 - 1/p\), or even more generally \(sp = \dimdom -1\), that the trace works well in the framework of functions of vanishing mean oscillation \citelist{\cite{Brezis_Nirenberg_1996}\cite{Mazowiecka_VanSchaftingen_2023}*{prop.\ 2.8}}.

\medskip

The final part of this work starts from the \emph{Bourgain-Brezis-Mironescu characterisation of Sobolev spaces} \citelist{\cite{Bourgain_Brezis_Mironescu_2001}\cite{Brezis_2002}\cite{Brezis_Mironescu_2021}*{th.\ 6.2}}:
if \(u \in \sobolev^{1, p} \brk{\Rset^{\dimdom}} \cap L^p \brk{\Rset^{\dimdom}}\), then
\begin{equation}
\label{eq_zaewahnoiKung3quiePhie5u}
 \lim_{s \underset{<}{\to}  1}\;
 \brk{1 - s}
 \smashoperator{\iint_{\Rset^{\dimdom} \times \Rset^{\dimdom}}}
 \frac{\abs{u \brk{x} - u \brk{y}}^p}{\abs{x - y}^{\dimdom + sp}}
 \dif x \dif y
 = \frac{2{} \pi^{\frac{\dimdom - 1}{2}} \Gamma \brk[\big]{\frac{p + 1}{2}}}
{p \Gamma \brk[\big]{\frac{\dimdom + p}{2}}}
\int_{\Rset^{\dimdom}} \abs{\Deriv u}^p.
\end{equation}
and conversely, if the function \(u \colon \Rset^{\dimdom} \to \Rset\) is measurable and if
\[
  \liminf_{s \underset{<}{\to} 1}\;
  \brk{1 - s}
 \smashoperator{\iint_{\Rset^{\dimdom} \times \Rset^{\dimdom}}}
 \frac{\abs{u \brk{x} - u \brk{y}}^p}{\abs{x - y}^{\dimdom + sp}}
 \dif x \dif y < \infty,
\]
then \(u \in \sobolev^{1, p} \brk{\Rset^{\dimdom}}\) and
\[
 \frac{2{} \pi^{\frac{\dimdom - 1}{2}} \Gamma \brk{\frac{p + 1}{2}}}
{p \Gamma \brk{\frac{\dimdom + p}{2}}}
\int_{\Rset^{\dimdom}} \abs{\Deriv u}^p
\le  \liminf_{s \underset{<}{\to} 1}\;
  \brk{1 - s}
 \smashoperator{\iint_{\Rset^{\dimdom} \times \Rset^{\dimdom}}}
 \frac{\abs{u \brk{x} - u \brk{y}}^p}{\abs{x - y}^{\dimdom + sp}}
 \dif x \dif y < \infty.
\]

For trace conjunction integrals, we prove that if \(u \in C^1_c \brk{\Bar{\Rset}^{\dimdom}_+}\), then (\cref{theorem_BBM_formula})
\begin{equation}
\label{eq_emiaQuohtohseitaeph1ahP4}
 \lim_{s \underset{<}{\to} 1}\; \brk{1 - s}
 \int_{\partial \Rset^{\dimdom}_+} \brk[\bigg]{\int_{\Rset^{\dimdom}_+} \frac{\abs{u \brk{x} - u \brk{y}}^p}{\abs{x - y}^{\dimdom + sp}} \dif y} \dif x
 =
 \frac{\pi^{\frac{\dimdom - 1}{2}} \Gamma \brk[\big]{\frac{p + 1}{2}}}
{p \Gamma \brk[\big]{\frac{\dimdom + p}{2}}}
 \int_{\partial \Rset^{\dimdom}_+} \abs{\Deriv u}^p.
\end{equation}
Compared to the classical Bourgain-Brezis-Mironescu formula \eqref{eq_zaewahnoiKung3quiePhie5u}, the right-hand side of its trace conjunction counterpart \eqref{eq_emiaQuohtohseitaeph1ahP4} features an integral that is only performed on the  \emph{boundary} but an integrand involving both the \emph{tangential} and \emph{normal} components of the derivative.
The identity \ref{eq_emiaQuohtohseitaeph1ahP4} is currently only proved for \emph{smooth} functions.

\begin{openproblem}
Does \eqref{eq_emiaQuohtohseitaeph1ahP4} still hold when \(u\) belongs to some suitable Sobolev space?
\end{openproblem}

Thanks to the relationship between the conjunction integral and the Gagliardo energy, we also prove in \cref{theorem_BBM_liminf} that if \(p > 1\) and if
\begin{equation}
\label{eq_ohgai8lie7nooGhi5iW2Xoh3}
  \liminf_{s \underset{<}{\to} 1}\; \brk{1 - s}
 \int_{\partial \Rset^{\dimdom}_+} \brk[\bigg]{\int_{\Rset^{\dimdom}_+ } \frac{\abs{v \brk{x} - u \brk{y}}^p}{\abs{x - y}^{\dimdom + sp}} \dif y} \dif x
 < \infty,
\end{equation}
then one has \(v \in \sobolev^{1, p}\brk{\partial \Rset^{\dimdom}_+}\); following \citelist{\cite{Davila_2002}\cite{VanSchaftingen_Willem_2004}\cite{Ponce_2004}}, when \(p = 1\) and \eqref{eq_ohgai8lie7nooGhi5iW2Xoh3} holds, the function \(v\) lies in the space of functions of bounded variation \(BV \brk{\partial \Rset^{\dimdom}_+}\) (\cref{theorem_BBM_liminf_BV}).

We do not expect these results to be complete, as they do not provide any of the information on the normal derivative that the identity \eqref{eq_emiaQuohtohseitaeph1ahP4} suggests.
This raises the question about more precise results.

\begin{openproblem}
Prove that the condition \eqref{eq_usohVeh9siqui5Zahahtae4X}
implies the existence of a normal derivative of \(u\) in some weak sense.
\end{openproblem}

At the other endpoint of the range \(s \in \intvo{0}{1}\) it would also make sense to study the limit
\[
  \lim_{s \underset{>}{\to} 0}\;
  s\int_{\partial \Rset^{\dimdom}_+} \brk[\bigg]{\int_{\Rset^{\dimdom}_+} \frac{\abs{u \brk{x} - u \brk{y}}^p}{\abs{x - y}^{\dimdom + sp}} \dif y} \dif x,
\]
following the \emph{Maz\/\cprime{}\!ya-Shaposhnikova formula} for fractional Sobolev spaces \cite{Mazya_Shaposhnikova_2002}.

Further interesting problems involve studying the counterparts of the conjunction integral inspired by other
the characterisation of Sobolev spaces.
The \emph{Nguyen formula} \cite{Nguyen_2006} (see also \citelist{\cite{Brezis_Nguyen_2018}\cite{Bourgain_Nguyen_2006}\cite{Nguyen_2008}})
\begin{equation}
\label{eq_Athe3xohyaen7ohJu2eiX8Za}
\lim_{\delta \to 0}
 \smashoperator[r]{\iint_{\substack{x, y \in \Rset^{\dimdom}\\ \abs{u \brk{x} - u \brk{y}}\ge \delta }}}
 \frac{\delta^p}{\abs{x - y}^{\dimdom + p}} \dif x \dif y
 =
 \frac{2 \pi^{\frac{\dimdom - 1}{2}} \Gamma \brk{\frac{p + 1}{2}}}
{p \Gamma \brk{\frac{\dimdom + p}{2}}}
\int_{\Rset^{\dimdom}} \abs{\Deriv u}^p,
\end{equation}
hints at investigating the limit
\begin{equation}
\lim_{\delta \to 0}
 \smashoperator[r]{\iint_{\substack{x \in \partial \Rset^\dimdom_+, y \in \Rset^{\dimdom}_+\\ \abs{v \brk{x} - u \brk{y}}\ge \delta}}}
 \frac{\delta^p}{\abs{x - y}^{\dimdom + p}} \dif x \dif y.
\end{equation}
Similarly,
the \emph{Brezis-Van Schaftingen-Yung formula} \cite{Brezis_VanSchaftingen_Yung_2021_PNAS} (see also \cite{Brezis_VanSchaftingen_Yung_2021_CVPDE})
\begin{equation}
\label{eq_onohzeeyi2eep3ootee5JaeT}
\lim_{\lambda \to \infty}
\lambda^p\, \abs[\bigg]{\set[\bigg]{\brk{x, y} \in \Rset^{\dimdom} \times \Rset^{\dimdom}\st
\frac{\abs{u \brk{x} - u \brk{y}}^p}{\abs{x - y}^{\dimdom + p}} \ge \lambda^p }}
= \frac{2\pi^{\frac{\dimdom - 1}{2}} \Gamma \brk{\frac{p + 1}{2}}}
{\dimdom\, \Gamma \brk{\frac{\dimdom + p}{2}}}
\int_{\Rset^{\dimdom}} \abs{\Deriv u}^p
.
\end{equation}
suggests studying the limit
\begin{equation}
\lim_{\lambda \to \infty}
\lambda^p\, \abs[\bigg]{\set[\bigg]{\brk{x, y} \in \partial \Rset^{\dimdom}_+ \times \Rset^{\dimdom}_+ \st
\frac{\abs{v \brk{x} - u \brk{y}}^p}{\abs{x - y}^{\dimdom + p}} \ge \lambda^p }}.
\end{equation}
Even more generally, it would make sense to study conjunction integral counterparts for the full family of \emph{Brezis-Seeger-Van Schaftingen-Yung formulae} that includes \eqref{eq_Athe3xohyaen7ohJu2eiX8Za} and \eqref{eq_onohzeeyi2eep3ootee5JaeT} \citelist{\cite{Brezis_Seeger_VanSchaftingen_Yung_2022}\cite{Brezis_Seeger_VanSchaftingen_Yung_2024}} (see also \cite{Dominguez_Seeger_Street_VanSchaftingen_Yung_2023}).

\medskip

The author thanks Petru Mironescu, Vitaly Moroz and Po-Lam Yung for enlightening discussions.

\section{First-order trace conjunction inequality}
Wotivated by Uspenskiĭ's  weighted version of the trace inequality \eqref{eq_Epieshu8ohr2iecuewei1eed} \cite{Uspenskii_1961} (see also \cite{Mironescu_Russ_2015}):
if the function \(u \colon \Rset^{\dimdom}_+ \to \Rset\) is weakly differentiable and if
\[
 \int_{\Rset^{\dimdom}_+} \frac{\abs{\Deriv u \brk{z}}^p}{
 z_{\dimdom}^{1  - (1 -s)p} } \dif z < \infty,
\]
then \(u\) has a trace \(v = \trace_{\partial \Rset^{\dimdom}_+} u\) and
\begin{equation}
\label{eq_phaeW8Ualoh6kaiph3ahwaeg}
 \smashoperator{\iint_{\partial \Rset^{\dimdom}_+ \times \partial \Rset^{\dimdom}_+ }}
 \frac{\abs{v \brk{x} - v \brk{y}}^p}
 {\abs{x - y}^{\dimdom - 1+ sp}}
 \dif y
 \dif x
 \le
 C
 \int_{\Rset^{\dimdom}_+} \frac{\abs{\Deriv u \brk{z}}^p}{
 z_{\dimdom}^{1  - (1 -s)p} } \dif y,
\end{equation}
we will prove the following trace conjunction inequality \eqref{eq_eexuevaeyeaghooPhie6ohwa}.

\begin{theorem}
\label{theorem_first_order_trace_conjunction}
If \(\dimdom \in \Nset \setminus \set{0, 1}\), if \(s \in \intvo{0}{1}\) and if \(p \in \intvo{1}{\infty}\), then every weakly differentiable function \(u \colon \Rset^{\dimdom}_+ \to \Rset\) satisfying
\[
  \int_{\Rset^{\dimdom}_+} \frac{\abs{\Deriv u \brk{z}}^p}{
 z_{\dimdom}^{1  - (1 -s)p} } \dif z < \infty
\]
has a trace \(v = \trace_{\partial \Rset^{\dimdom}_+} u\) satisfying
\begin{equation}
\label{eq_EMohhuduothaey8xee9Oshi1}
 \int_{\partial \Rset^{\dimdom}_+}
 \brk[\bigg]{\int_{\Rset^{\dimdom}_+}
 \frac{\abs{v \brk{x} - u \brk{y}}^p}
 {\abs{x - y}^{\dimdom + sp}}
 \dif y}
 \dif x
 \le
 \frac{\pi^\frac{\dimdom -1}{2}\Gamma \brk{\frac{sp + 1}{2}}}{\Gamma \brk{\frac{\dimdom + sp}{2}}} \brk[\bigg]{\frac{\dimdom}{s}}^p
 \int_{\Rset^{\dimdom}_+} \frac{\abs{\Deriv u \brk{z}}^p}{
 z_{\dimdom}^{1  - (1 -s)p} } \dif y.
\end{equation}
\end{theorem}

The main ingredient in the proof \cref{theorem_first_order_trace_conjunction} will be the following weighted Hardy inequality.

\begin{proposition}
\label{proposition_Hardy}
If \(\dimdom \in \Nset \setminus \set{0, 1}\), if \(s \in \intvo{0}{1}\), if \(p \in \intvo{1}{\infty}\) and if \(u \in C^1\brk{\Bar{\Rset}^\dimdom_+}\), then
\begin{equation}
\label{eq_AeKeefee6ieGhoTac5aezohr}
  \int_{\Rset^\dimdom_+} \frac{\abs{u \brk{x} - u \brk{0}}^p}{\abs{x}^{\dimdom + sp}}  \dif x
  \le \brk[\bigg]{\frac{\dimdom}{s}}^p
  \int_{\Rset^\dimdom_+}  \frac{\abs{\Deriv u\brk{x}}^p
  x_\dimdom^p}{\abs{x}^{\dimdom + sp}} \dif x.
\end{equation}
\end{proposition}

In the classical Hardy inequality, the left-hand side of \eqref{eq_AeKeefee6ieGhoTac5aezohr} is controlled as
\begin{equation}
\label{eq_aim7reifatei2booG5aubee2}
 \int_{\Rset^\dimdom_+} \frac{\abs{u \brk{x} - u \brk{0}}^p}{\abs{x}^{\dimdom + sp}}  \dif x
 \le \frac{1}{s^p}
 \int_{\Rset^{\dimdom}_+} \frac{\abs{\Deriv u\brk{x}}^p}{\abs{x}^{\dimdom - \brk{1 -  s}p}} \dif x;
\end{equation}
the integrand on the right-hand side in  \eqref{eq_AeKeefee6ieGhoTac5aezohr} is smaller than in \eqref{eq_aim7reifatei2booG5aubee2}; its stronger decay in the tangential direction will be crucial in the proof of \cref{theorem_first_order_trace_conjunction}.

\begin{proof}[Proof of \cref{proposition_Hardy}]
Defining first the vector field \(\xi \colon \Rset^{\dimdom}_+ \to \Rset^{\dimdom}\) at every point \(x = \brk{x', x_{\dimdom}} \in \Rset^{\dimdom}_+ = \Rset^{\dimdom - 1} \times \intvo{0}{\infty}\) by
\begin{equation}
\label{eq_Ue8aa5iep2ath1cee6eecei5}
  \xi \brk{x}
  \defeq
  \frac{x_{\dimdom} e_{\dimdom}}{\abs{x}^{\dimdom + sp}}
  -
  \brk[\Big]{1 + \frac{\dimdom}{sp}} \frac{x_{\dimdom}^2 x}{\abs{x}^{\dimdom + sp + 2}},
\end{equation}
where \(e_1, \dotsc, e_{\dimdom}\) is the canonical basis of \(\Rset^{\dimdom}\),
we compute for every \(x \in \Rset^{\dimdom}_+\)
\begin{equation}
\label{eq_zeejir0Giegie6thu6me1loo}
\begin{split}
 \operatorname{div} \xi \brk{x}
 &= \frac{1}{\abs{x}^{\dimdom + sp}} -
  \brk{\dimdom + sp} \frac{x_\dimdom^2}{\abs{x}^{\dimdom + sp}}\\
 &\qquad \qquad  - \brk[\Big]{1 + \frac{\dimdom}{sp}}\brk{2 + \dimdom - \brk{\dimdom + sp + 2}}
 \frac{x_{\dimdom}^2 }{\abs{x}^{\dimdom + sp + 2}}\\
 &= \frac{1}{\abs{x}^{\dimdom + sp}}.
\end{split}
\end{equation}
We also have from \eqref{eq_Ue8aa5iep2ath1cee6eecei5}
for every \(x \in \Rset^{\dimdom}_+\)
\[
 \abs{\xi \brk{x}}^2 = \frac{x_{\dimdom}^2}{\abs{x}^{2 \brk{\dimdom + sp}}}+
 \brk[\bigg]{\brk[\bigg]{\frac{\dimdom}{sp}}^2 - 1} \frac{x_{\dimdom}^4}{\abs{x}^{2 \brk{\dimdom + sp + 1}}}
 \le \brk[\bigg]{\frac{\dimdom}{sp}}^2 \frac{x_\dimdom^2}{\abs{x}^{2 \brk{\dimdom + sp}}},
\]
so that
\begin{equation}
\label{eq_MoShaiQu0mah5ootei8ohsae}
 \abs{\xi \brk{x}}
 \le \frac{\dimdom}{sp} \frac{x_\dimdom}{\abs{x}^{\dimdom + sp}},
\end{equation}
and
\begin{equation}
\label{eq_ohrootahgoo8apooGhae2ohV}
 \xi \brk{x} \cdot x = -\frac{\dimdom x_{\dimdom}^2}{sp \abs{x}^{\dimdom + sp}}.
\end{equation}

We then get for each \(R > r > 0\), by the divergence theorem
\begin{equation}
\label{eq_pothah9aeJohcei2oazohshu}
\begin{split}
  &\int_{\brk{\Bset^{\dimdom}_R\! \brk{0} \setminus \Bset^{\dimdom}_r\! \brk{0}}\cap \Rset^\dimdom_+} \abs{u \brk{x} - u \brk{0}}^p \operatorname{div} \xi \brk{x} \dif x\\
  &\qquad =   - p \int_{\brk{\Bset^{\dimdom}_R\! \brk{0} \setminus \Bset^{\dimdom}_r\! \brk{0}}\cap \Rset^{\dimdom}_+} \abs{u \brk{x} - u \brk{0}}^{p - 1} \Deriv u \brk{x} \sqb{\xi \brk{x}} \dif x \\
  &\qquad \qquad + \int_{\partial \Bset^{\dimdom}_R\! \brk{0} \cap \Rset^{\dimdom}_+} \abs{u \brk{x} - u \brk{0}}^p \frac{\xi \brk{x} \cdot x}{\abs{x}}\dif x\\
  &\qquad \qquad - \int_{\partial \Bset^{\dimdom}_r\! \brk{0} \cap \Rset^{\dimdom}_+} \abs{u \brk{x} - u \brk{0}}^p \frac{\xi \brk{x} \cdot x}{\abs{x}}\dif x.
\end{split}
\end{equation}
The first term in the right-hand side of \eqref{eq_pothah9aeJohcei2oazohshu} is controlled by Young's inequality and by \eqref{eq_MoShaiQu0mah5ootei8ohsae} as
\begin{equation}
\label{eq_Ungouchu9eeh5thuJaedeiv4}
\begin{split}
 &-p\int_{\brk{\Bset^{\dimdom}_R\! \brk{0} \setminus \Bset^{\dimdom}_r\! \brk{0}}\cap \Rset^{\dimdom}_+} \abs{u \brk{x} - u \brk{0}}^{p - 1} \Deriv u \brk{x} \sqb{\xi \brk{x}} \dif x \\
 &\quad \le \brk[\Big]{1 - \frac{1}{p}} \int_{\brk{\Bset^{\dimdom}_R\! \brk{0} \setminus \Bset^{\dimdom}_r\! \brk{0}} \cap \Rset^\dimdom_+} \frac{\abs{u \brk{x} - u \brk{0}}^p}{\abs{x}^{\dimdom + sp}}  \dif x
 + \frac{1}{p}\brk[\bigg]{\frac{\dimdom}{s}}^p \int_{\Rset^\dimdom_+}  \frac{\abs{\Deriv u\brk{x}}^p
  x_\dimdom^p}{\abs{x}^{\dimdom + sp}}
  \dif x.
\end{split}
\end{equation}
The second term in the right-hand side of \eqref{eq_pothah9aeJohcei2oazohshu} is nonnegative in view of \eqref{eq_ohrootahgoo8apooGhae2ohV}:
\begin{equation}
\label{eq_phuweingohlug4pa0eiHieyu}
 \int_{\partial \Bset^{\dimdom}_R\! \brk{0} \cap \Rset^{\dimdom}_+} \abs{u \brk{x} - u \brk{0}}^p \frac{\xi \brk{x} \cdot x}{\abs{x}}\dif x \le 0.
\end{equation}
The last term in the the right-hand side of \eqref{eq_pothah9aeJohcei2oazohshu} can be controlled thanks to \eqref{eq_ohrootahgoo8apooGhae2ohV} again as
\begin{equation}
\label{eq_Eyohnaeyu8doofahYeiHo8ph}
\begin{split}
- \int_{\partial \Bset^{\dimdom}_r\! \brk{0} \cap \Rset^{\dimdom}_+} &\abs{u \brk{x} - u \brk{0}}^p \frac{\xi \brk{x} \cdot x}{\abs{x}}\dif x\\
&=
 \frac{\dimdom}{s}\int_{\partial \Bset^{\dimdom}_r\! \brk{0} \cap \Rset^{\dimdom}_+}
  \frac{\abs{u \brk{x} - u\brk{0}}^p x_{\dimdom}^2}{\abs{x}^{\dimdom + sp + 1}} \dif x\\
&\le \frac{\dimdom}{sp} \norm{\Deriv u}_{L^\infty \brk{\Rset^\dimdom_+}}^p
\int_{\partial \Bset^{\dimdom}_\dimdom \cap \Rset^{\dimdom}_+}
\frac{1}{\abs{x}^{\dimdom - \brk{1 - s} p - 1}} \dif x\\
&\le \frac{\dimdom \abs{\partial \Bset^{\dimdom}_1}}{2 sp} \norm{\Deriv u}_{L^\infty \brk{\Rset^\dimdom_+}}^p r^{\brk{1 - s}p}.
  \end{split}
\end{equation}
It follows thus from \eqref{eq_pothah9aeJohcei2oazohshu},  \eqref{eq_Ungouchu9eeh5thuJaedeiv4}, \eqref{eq_phuweingohlug4pa0eiHieyu} and \eqref{eq_Eyohnaeyu8doofahYeiHo8ph} that
\begin{equation}
\label{eq_Airuecheu2ziiphai9fa5bei}
\begin{split}
 &\int_{\brk{\Bset^{\dimdom}_R\! \brk{0} \setminus \Bset^{\dimdom}_r\! \brk{0}}\cap \Rset^\dimdom_+} \frac{\abs{u \brk{x} - u \brk{0}}^p}{\abs{x}^{\dimdom + sp}} \dif x\\
 &\qquad \le \brk[\bigg]{\frac{\dimdom}{s}}^p \int_{\Rset^\dimdom_+}  \frac{\abs{\Deriv u\brk{x}}^p
  x_\dimdom^p}{\abs{x}^{\dimdom + sp}}
  \dif x
  + \frac{\dimdom \abs{\partial \Bset^{\dimdom}_1}}{2 s} \norm{\Deriv u}_{L^\infty \brk{\Rset^\dimdom_+}}^p r^{\brk{1 - s}p}.
  \end{split}
\end{equation}
Letting \(R \to \infty\) and \(r \to 0\) in \eqref{eq_Airuecheu2ziiphai9fa5bei} we reach the conclusion \eqref{eq_AeKeefee6ieGhoTac5aezohr}
since \(s < 1\).
\end{proof}

The constant in the estimate \eqref{eq_EMohhuduothaey8xee9Oshi1} of \cref{theorem_first_order_trace_conjunction} will also rely on the following computation.
\begin{lemma}
\label{lemma_potential_integral}
If \(\dimdom \in \Nset \setminus \set{0, 1}\), \(s \in \intvo{0}{1}\), \(p \in \intvr{1}{\infty}\) and \(x \in \Rset^{\dimdom}_+\), then
\[
 \int_{\Rset^{\dimdom - 1}} \frac{1}{\abs{x - y}^{\dimdom + sp}} \dif y
  = \frac{\pi^\frac{\dimdom -1}{2}\Gamma \brk{\frac{sp + 1}{2}}}{\Gamma \brk{\frac{\dimdom + sp}{2}} x_{\dimdom}^{sp + 1}}.
\]
\end{lemma}
\begin{proof}
Integrating in spherical coordinates and applying the change of variable \(r = x_{\dimdom} \sqrt{t^{-1} - 1}\), we get
\[
\begin{split}
  \int_{\Rset^{\dimdom - 1}} \frac{1}{\abs{x - y}^{\dimdom + sp}} \dif y
  &= \frac{2 \pi^{\frac{\dimdom - 1}{2}}}{\Gamma \brk{\frac{\dimdom  - 1}{2}}}
  \int_0^\infty \frac{r^{\dimdom - 2}}{\brk{x_{\dimdom}^2 + r^2}^\frac{\dimdom + sp}{2}} \dif r\\
  &= \frac{\pi^{\frac{\dimdom - 1}{2}}}{\Gamma \brk{\frac{\dimdom  - 1}{2}}x_{\dimdom}^{sp + 1}}
  \int_0^1 \brk{1 - t}^{\frac{\dimdom - 1}{2} - 1} t^{\frac{sp + 1}{2} - 1}\dif t\\
  & = \frac{\pi^{\frac{\dimdom - 1}{2}}\mathrm{B} \brk{\tfrac{\dimdom - 1}{2}, \tfrac{sp + 1}{2}}}{\Gamma \brk{\frac{\dimdom  - 1}{2}} x_{\dimdom}^{sp + 1}} =  \frac{\pi^{\frac{\dimdom - 1}{2}}\Gamma \brk{\frac{sp + 1}{2}}}{\Gamma \brk{\frac{\dimdom  +sp}{2}}x_{\dimdom}^{sp + 1}},
\end{split}
\]
in view of the volume formula for the unit sphere \(\partial \Bset^{\dimdom - 1}_1\) and the properties of the Beta function.
\end{proof}

We can now prove \cref{theorem_first_order_trace_conjunction}.
.
\begin{proof}[Proof of \cref{theorem_first_order_trace_conjunction}]
By density we can assume that \(u \in C^1 \brk{\Bar{\Rset}^\dimdom_+}\).
By \cref{proposition_Hardy} and \cref{lemma_potential_integral}, we have then
\[
\begin{split}
\int_{\partial \Rset^{\dimdom}_+}
\brk[\bigg]{
 \int_{\Rset^{\dimdom}_+}
    \frac{\abs{\Deriv u\brk{x}}^p}{\abs{x - y}^{\dimdom + sp}} \dif y }\dif x
    &\le \brk[\bigg]{\frac{\dimdom}{s}}^p
    \int_{\partial \Rset^{\dimdom}_+}
    \brk[\bigg]{
  \int_{\Rset^\dimdom_+}  \frac{\abs{\Deriv u\brk{y}}^p
  y_\dimdom^p}{\abs{x - y}^{\dimdom + sp}} \dif y} \dif x \\
& = \brk[\bigg]{\frac{\dimdom}{s}}^p
    \int_{\Rset^{\dimdom}_+}
    \brk[\bigg]{
  \int_{\partial \Rset^\dimdom_+}  \frac{\abs{\Deriv u\brk{y}}^p
  y_\dimdom^p}{\abs{x - y}^{\dimdom + sp}} \dif x} \dif y \\
    &= \frac{\pi^\frac{\dimdom -1}{2}\Gamma \brk{\frac{sp + 1}{2}}}{\Gamma \brk{\frac{\dimdom + sp}{2}}} \brk[\bigg]{\frac{\dimdom}{s}}^p
    \int_{\Rset^\dimdom_+} \frac{\abs{\Deriv u \brk{y}}^p}{y_{\dimdom}^{1 - \brk{1 - s}p}} \dif  y,
\end{split}
\]
which proves the announced inequality \eqref{eq_EMohhuduothaey8xee9Oshi1}.
\end{proof}

\section{Fractional trace conjunction inequality}

The fractional Sobolev spaces not only appear as traces of first-order Sobolev spaces, but also have traces on their own.
Indeed \citelist{\cite{Besov_1961}\cite{Taibleson_1964}*{th.\ 12}} (see also \cite{Leoni_2023}*{th.\ 9.14}), if \(sp > 1\), every \(u \in \sobolev^{s, p} \brk{\Rset^{\dimdom}_+}\) has  trace
\(v = \trace_{\partial \Rset^{\dimdom}_+} u
\in \sobolev^{s - 1/p, p} \brk{\partial \Rset^{\dimdom}_+}\) satisfying
\begin{equation}
\label{eq_Jeengaifai8aivaeghiphuye}
 \smashoperator[r]{\iint_{\partial \Rset^{\dimdom}_+ \times \partial \Rset^{\dimdom}_+}}
 \frac{\abs{v \brk{x} -  v \brk{y}}^p}
 {\abs{x - y}^{\dimdom - 2 + sp}}
 \dif x
 \dif y
 \le
 C \smashoperator{\iint\limits_{\Rset^{\dimdom}_+ \times \Rset^{\dimdom}_+}}
 \frac{\abs{u \brk{y} - u \brk{z}}^p}{\abs{y - z}^{\dimdom + sp}} \dif y \dif z.
\end{equation}
We prove a corresponding trace conjunction inequality.

\begin{theorem}
\label{theorem_fractional_trace_conjunction}
For every \(\dimdom \in \Nset \setminus \set{0, 1}\), \(s \in \intvo{0}{1}\) and \(p \in \intvr{1}{\infty}\) satisfying \(sp > 1\), there exists a constant \(C \in \intvo{0}{\infty}\) such every \(u \in \sobolev^{s, p} (\Rset^{\dimdom}_+)\) has a trace \(v = \trace_{\partial \Rset^{\dimdom}_+} u\) satisfying
\begin{equation}
\label{eq_ooNgohhaeBaeshohfei2aiW7}
\int_{\partial \Rset^{\dimdom}_+}
 \brk[\bigg]{\int_{\Rset^{\dimdom}_+}
 \frac{\abs{v \brk{x} - u \brk{y}}^p}
 {\abs{x - y}^{\dimdom - 1 + sp}}
 \dif y}
 \dif x
 \le C \smashoperator{\iint_{\Rset^{\dimdom}_+ \times \Rset^{\dimdom}_+}}
 \frac{\abs{u \brk{y} - u \brk{z}}^p}{\abs{y - z}^{\dimdom + sp}} \dif y \dif z.
\end{equation}
\end{theorem}

Even though the proof of \cref{theorem_fractional_trace_conjunction} will not be based on a Hardy inequality, our proof will rely on a the strategy devised by Brezis, Mironescu and Ponce for fractional Hardy inequalities \citelist{\cite{Mironescu_2018}*{Lemma 2}\cite{Brezis_Mironescu_Ponce_2004}}.

\begin{proof}%
[Proof of \cref{theorem_fractional_trace_conjunction}]
By density,
we can assume that \(u \in C^1_c \brk{\Bar{\Rset}^{\dimdom}_+}\) so that in particular we have
\begin{equation}
\label{eq_iekai9dai6Aa9YieDitheeLu}
 \int_{\partial \Rset^{\dimdom}_+}
 \brk[\bigg]{\int_{\Rset^{\dimdom}_+}
 \frac{\abs{v \brk{x} - u \brk{y}}^p}
 {\abs{x - y}^{\dimdom - 1 + sp}}
 \dif y}
 \dif x < \infty,
\end{equation}
with \(v = u \vert_{\partial \Rset^{\dimdom}_+}\).
Defining for \(\lambda > 0\) the set
\begin{equation}
 A_\lambda
 \defeq
 \set[\big]{(x, y, z) \in \partial \Rset^{\dimdom}_+ \times \Rset^{\dimdom}_+ \times \Rset^{\dimdom}_+
 \st \abs{z - x} \le \lambda \abs{x - y}},
\end{equation}
we have by convexity and the triangle inequality
\begin{equation}
\label{eq_ahx3ieree4xoe5ChuV4IeXes}
\begin{split}
 \iiint\limits_{A_\lambda} \frac{\abs{u \brk{x} - u \brk{y}}^p}{\abs{x - y}^{2\dimdom - 1 + sp}} \dif x \dif y \dif z
 &\le
 2^{p - 1}\iiint\limits_{A_\lambda} \frac{ \abs{u \brk{x} - u \brk{z}}^p}{\abs{x - y}^{2\dimdom - 1 + sp}} \dif x \dif y \dif z \\
 &\qquad +
 2^{p - 1} \iiint\limits_{A_\lambda} \frac{\abs{u \brk{y} - u \brk{z}}^p}{\abs{x - y}^{2\dimdom - 1 + sp}} \dif x \dif y \dif z .
\end{split}
\end{equation}
For the left-hand side
of \eqref{eq_ahx3ieree4xoe5ChuV4IeXes},
 we note that for each \(x \in \partial \Rset^{\dimdom}_+\) and \(y \in \Rset^{\dimdom}_+\),
\begin{equation}
\label{eq_xeesi1ohheiphau3Xeex1Jee}
 \smashoperator[r]{\int_{\substack{z \in \Rset^{\dimdom}_+\\
 \abs{z - x} \le \lambda \abs{x - y}}}}
 \frac{1}{\abs{x - y}^{2\dimdom - 1 + sp}} \dif z
 =
 \frac{\lambda^{\dimdom}\abs{\Bset^{\dimdom}_1}}{2 \abs{x - y}^{\dimdom - 1 + sp}},
\end{equation}
whereas for the first term in the right-hand side of
\eqref{eq_ahx3ieree4xoe5ChuV4IeXes} we have
\begin{equation}
\label{eq_coh7uak5thee6iejaiP1dibe}
\begin{split}
\smash{ \smashoperator[r]{\int_{\substack{y \in \Rset^{\dimdom}_+\\
 \abs{z - x} \le \lambda \abs{x - y}}}}
 \frac{1}{\abs{x - y}^{2\dimdom - 1 + sp}} \dif y}
 &= \frac{\abs{\partial \Bset^{\dimdom}_1}}{2}
 \int_{\abs{z - x}/\lambda}^\infty \frac{1}{r^{\dimdom + sp}}  \dif r\\
 &= \frac{\dimdom \abs{\Bset^{\dimdom}_1} \lambda^{\dimdom - 1 + sp}}{2\brk{\dimdom - 1 + sp} \abs{z - x}^{\dimdom - 1 + sp}}.
\end{split}
\end{equation}
In order to treat the second term in the right-hand side of
\eqref{eq_ahx3ieree4xoe5ChuV4IeXes}, we note that if \(\abs{z - x} \le \lambda \abs{x - y}\), then by the triangle inequality
\begin{equation}
 (1 - \lambda) \abs{x - y}
 \le
 \abs{x - y} - \abs{z - x}
 \le \abs{y - z}
 \le \abs{x - y} + \abs{z - x}
 \le (1 + \lambda) \abs{x - y},
\end{equation}
and hence, if \(\lambda < 1\)
\begin{equation}
\label{eq_IFiaGuph0aelo1aiquei2eeL}
\begin{split}
 \smashoperator[r]{\int_{\substack{x \in \partial \Rset^{\dimdom}_+\\
 \abs{z - x} \le \lambda \abs{x - y}}}}
 \;
 \frac{1}{\abs{x - y}^{2\dimdom - 1 + sp}} \dif x
 &\le
 \smashoperator[r]{
 \int\limits_{\substack{x \in \partial \Rset^{\dimdom}_+\\
 \brk{1 - \lambda} \abs{x - y} \le \abs{y - z}}}}
 \;
 \frac{\brk{1 + \lambda}^{2\dimdom - 1 + sp}}{\abs{y - z}^{2\dimdom - 1 + sp}} \dif x\\
 &\le \abs{\Bset^{\dimdom - 1}_1} \frac{\brk{1 + \lambda}^{2\dimdom - 1 + sp} }{\brk{1 - \lambda}^{\dimdom - 1}\abs{y - z}^{\dimdom + sp}}.
\end{split}
\end{equation}
Inserting \eqref{eq_xeesi1ohheiphau3Xeex1Jee},
\eqref{eq_coh7uak5thee6iejaiP1dibe} and  \eqref{eq_IFiaGuph0aelo1aiquei2eeL} into \eqref{eq_ahx3ieree4xoe5ChuV4IeXes} we get, in view of \eqref{eq_iekai9dai6Aa9YieDitheeLu},
\begin{multline}
\label{eq_iughei9itheaGouyahpeitha}
\brk[\bigg]{\frac{1}{2^{p - 1}} - \frac{\dimdom \lambda^{sp - 1}}{\dimdom -1 + sp}}
\int_{\partial \Rset^{\dimdom}_+}
 \brk[\bigg]{\int_{\Rset^{\dimdom}_+}
 \frac{\abs{u \brk{x} - u \brk{y}}^p}
 {\abs{x - y}^{\dimdom - 1 + sp}}
 \dif y}
 \dif x\\
 \le \frac{2\abs{\Bset^{\dimdom - 1}_1}}{\abs{\Bset^{\dimdom}_1}} \frac{\brk{1 + \lambda}^{2\dimdom - 1 + sp} }{\lambda^{\dimdom} \brk{1 - \lambda}^{\dimdom - 1}}
 \smashoperator[r]{\iint_{\Rset^{\dimdom}_+ \times \Rset^{\dimdom}_+}
 }\frac{\abs{u \brk{y} - u \brk{z}}^p}{\abs{y - z}^{\dimdom + sp}} \dif y \dif z.
\end{multline}
Since \(sp > 1\) we can fix \(\lambda \in \intvo{0}{1}\) so that \(\dimdom 2^{p - 1}\lambda^{sp - 1} < \dimdom - 1 + sp\) and get the estimate \eqref{eq_ooNgohhaeBaeshohfei2aiW7} as a consequence of \eqref{eq_iughei9itheaGouyahpeitha}.
\end{proof}

\section{Controling the Gagliardo integral by the conjunction integral}

The conjunction integral controls the Gagliardo integral so that the estimates \cref{theorem_first_order_trace_conjunction} and \cref{theorem_fractional_trace_conjunction} imply their classical counterparts \eqref{eq_phaeW8Ualoh6kaiph3ahwaeg} and \eqref{eq_Jeengaifai8aivaeghiphuye} respectively. 

\begin{theorem}
\label{theorem_mixed_to_gagliardo}
If \(\dimdom \in \Nset \setminus \set{0, 1}\), if \(s \in \intvo{0}{1}\), if \(p \in \intvr{1}{\infty}\) and if the functions \(u \colon \Rset^{\dimdom}_+ \to \Rset\) and \( v \colon \partial \Rset^{\dimdom}_+ \to \Rset\) are measurable, then
\begin{equation}
\begin{split}
 &\smashoperator[r]{\iint_{\partial \Rset^{\dimdom}_+ \times \partial \Rset^{\dimdom}_+}}
 \frac{\abs{v \brk{x} - v (y)}^p}{\abs{x - y}^{\dimdom - 1 + sp}} \dif x \dif y\\
 &\qquad\qquad
 \le
 \frac{2 \cdot 3^{2 \dimdom -1 + sp}\Gamma \brk{\frac{\dimdom + 2}{2}}}{4^{sp} \pi^\frac{1}{2} \Gamma \brk{\frac{\dimdom + 1}{2}}} \int_{\partial \Rset^{\dimdom}_+}
 \brk[\bigg]{\int_{\Rset^{\dimdom}_+}
 \frac{\abs{v \brk{x} - u \brk{y} }^p}
 {\abs{x - y}^{\dimdom + sp}}
 \dif y}
 \dif x.
 \end{split}
\end{equation}
\end{theorem}

\Cref{theorem_mixed_to_gagliardo} will follow from the more general result.

\begin{lemma}
\label{lemma_mixed_to_boundary}
If \(\dimdom \in \Nset \setminus \set{0, 1}\), if \(s \in \intvo{0}{1}\), if \(p \in \intvr{1}{\infty}\) and if the functions \(u \colon \Rset^{\dimdom}_+ \to \Rset\) and \( v, w \colon \partial \Rset^{\dimdom}_+ \to \Rset\) are measurable, then
\begin{equation}
\label{eq_eepie0aiCh9eesae8ji1Weel}
\begin{split}
 \smashoperator[r]{\iint\limits_{\partial \Rset^{\dimdom}_+ \times \partial \Rset^{\dimdom}_+}}
 \frac{\abs{v \brk{x} - w \brk{x}}^p}{\abs{x - z}^{\dimdom - 1 + sp}} \dif x \dif y
 &\le
  \frac{2 \cdot 3^{2 \dimdom -1 + sp}\Gamma \brk{\frac{\dimdom + 2}{2}}}{4^{sp} \pi^\frac{1}{2} \Gamma \brk{\frac{\dimdom + 1}{2}}}
  \brk[\Bigg]{
  \int_{\partial \Rset^{\dimdom}_+}
 \brk[\bigg]{\int_{\Rset^{\dimdom}_+}\hspace{-.5em}
 \frac{\abs{v \brk{x} - u \brk{y}}^p}
 {\abs{x - y}^{\dimdom + sp}}
 \dif y}
 \dif x
 \\
 &\hspace{4em} +
  \int_{\partial \Rset^{\dimdom}_+}
 \brk[\bigg]{\int_{\Rset^{\dimdom}_+}\hspace{-.5em}
 \frac{\abs{w \brk{z} - u \brk{y}}^p}
 {\abs{z - y}^{\dimdom + sp}}
 \dif y}
 \dif z}.
 \end{split}
\end{equation}
\end{lemma}
\begin{proof}
We define the set
\[
 A
 \defeq
 \set[\Big]{\brk{x, y, z} \in \partial \Rset^{\dimdom}_+ \times \Rset^{\dimdom}_+ \times \partial \Rset^{\dimdom}_+
 \st
 \abs{y - \tfrac{x + z}{2}} \le \tfrac{\abs{z - x}}{4}}
\]
and we write, thanks to the triangle inequality,
\begin{equation}
\label{eq_ua5Vohxi9oohohng6uise6oh}
\begin{split}
 \smashoperator{\iiint_{A}} \frac{\abs{v \brk{x} - w \brk{z}}^p}{\abs{x - z}^{2\dimdom - 1 + sp}} \dif x \dif y \dif z
 &\le 2^{p - 1} \iiint\limits_{A} \frac{\abs{v \brk{x} - u \brk{y}}^p}{\abs{x - z}^{2\dimdom - 1 + sp}} \dif x \dif y \dif z\\
 &\qquad + 2^{p - 1} \iiint\limits_{A} \frac{\abs{u \brk{y} - w \brk{z}}^p}{\abs{x - z}^{2\dimdom - 1 + sp}} \dif x \dif y \dif z.
 \end{split}
\end{equation}
We first compute
\begin{equation}
\label{eq_aCiighahmee0dahh1cheenie}
 \smashoperator[r]{\int\limits_{\substack{y \in \Rset^{\dimdom}_+\\ \abs{y - \frac{x + z}{2}} \le \frac{\abs{x - z}}{4}}}} \frac{1}{\abs{x - z}^{2\dimdom - 1 + sp}} \dif y
 \ge \frac{\abs{\Bset^{\dimdom}_1}}{2\cdot 4^N \abs{x - z}^{\dimdom - 1 + sp}}.
\end{equation}
Next, we have for each \(\brk{x, y, z} \in A\),
\[
\tfrac{1}{4} \abs{x - z}
\le
\abs{\tfrac{x - z}{2}} - \abs{y - \tfrac{x + z}{2}}
\le
   \abs{x - y}
   \le
   \abs{\tfrac{x - z}{2}} + \abs{y - \tfrac{x + z}{2}}
   \le
   \tfrac{3}{4} \abs{x - z},
\]
and therefore for every \(x \in \partial \Rset^{\dimdom}_+\)
and \(y \in \Rset^{\dimdom}_+\)
\begin{equation}
\label{eq_nai0che2fe5Eethohoo4goe1}
\begin{split}
\smashoperator[r]{\int\limits_{\substack{z \in \partial \Rset^{\dimdom}_+\\ \abs{y - \frac{x + y}{2}} \le \frac{\abs{x - z}}{4}}}}\; \frac{1}{\abs{x - z}^{2\dimdom - 1 + sp}} \dif z
 &
 \le \smashoperator[r]{\int\limits_{\substack{z \in \partial \Rset^{\dimdom}_+\\ \abs{x - z} \le 4 \abs{x - y}}}} \frac{3^{2\dimdom  - 1 + sp}}{4^{2\dimdom  - 1 + sp}\abs{x - y}^{2 \dimdom - 1 + sp}} \dif z
 \\
 &\le \frac{3^{2\dimdom - 1 + sp} \abs{\Bset^{\dimdom - 1}_1}}{4^{\dimdom + sp} \abs{x - y}^{\dimdom + sp}}.
 \end{split}
\end{equation}
Similarly, we have for every \(y \in \Rset^{\dimdom}_+\) and every \(z \in \partial \Rset^{\dimdom}_+\),
\begin{equation}
\label{eq_shah3pae1eeheeFa8phiefie}
\smashoperator[r]{\int\limits_{\substack{x \in \partial \Rset^{\dimdom}_+\\ \abs{y - \frac{x + z}{2}} \le \frac{\abs{x - z}}{4}}}}\;
\frac{1}{\abs{x - z}^{2\dimdom - 1 + sp}} \dif x
 \le \frac{3^{2\dimdom - 1 + sp} \abs{\Bset^{\dimdom - 1}_1}}{4^{\dimdom + sp} \abs{y - z}^{\dimdom + s p}}.
\end{equation}
Inserting \eqref{eq_aCiighahmee0dahh1cheenie}, \eqref{eq_nai0che2fe5Eethohoo4goe1} and \eqref{eq_shah3pae1eeheeFa8phiefie} into \eqref{eq_ua5Vohxi9oohohng6uise6oh}, we get
\[
\begin{split}
  \smashoperator[r]{\iint\limits_{\partial \Rset^{\dimdom}_+ \times \partial \Rset^{\dimdom}_+}}
 \frac{\abs{v \brk{x} - w \brk{x}}^p}{\abs{x - z}^{\dimdom - 1 + sp}} \dif x \dif y
 &\le
 \frac{2 \cdot 3^{2 \dimdom -1 + sp}\abs{\Bset^{\dimdom - 1}_1}}{\abs{\Bset^{\dimdom}_1}4^{sp}} \brk[\Bigg]{
  \int_{\partial \Rset^{\dimdom}_+}
 \brk[\bigg]{\int_{\Rset^{\dimdom}_+}
 \frac{\abs{v \brk{x} - u \brk{y}}^p}
 {\abs{x - y}^{\dimdom + sp}}
 \dif y}
 \dif x
 \\
 &\hspace{4em} +
  \int_{\partial \Rset^{\dimdom}_+}
 \brk[\bigg]{\int_{\Rset^{\dimdom}_+}
 \frac{\abs{w \brk{z} - u \brk{y}}^p}
 {\abs{y - z}^{\dimdom + sp}}
 \dif y}
 \dif z},
 \end{split}
\]
we get the conclusion \eqref{eq_eepie0aiCh9eesae8ji1Weel}.
\end{proof}
\section{Characterising the trace}
The next result shows that the finitiness of the trace conjunction integral characterises the trace.

\begin{theorem}
\label{theorem_trace_characterisation}
If \(\dimdom \in \Nset \setminus \set{0, 1}\), if \(s \in \intvo{0}{1}\), if \(p \in \intvr{1}{\infty}\), if the functions \(u \colon \Rset^{\dimdom}_+ \to \Rset\) and \( v, w \colon \partial \Rset^{\dimdom}_+ \to \Rset\) are measurable and if
\[
 \int_{\partial \Rset^{\dimdom}_+}
 \brk[\bigg]{\int_{\Rset^{\dimdom}_+}
 \frac{\abs{v \brk{x} - u \brk{y}}^p}
 {\abs{x - y}^{\dimdom + sp}}
 \dif y}
 \dif x < \infty
\]
and
\[
  \int_{\partial \Rset^{\dimdom}_+}
 \brk[\bigg]{\int_{\Rset^{\dimdom}_+}
 \frac{\abs{w \brk{x} - u \brk{y}}^p}
 {\abs{x - y}^{\dimdom + sp}}
 \dif y}
 \dif x < \infty,
\]
then \(v = w\) almost everywhere in \(\partial \Rset^{\dimdom}_+\).
\end{theorem}

The proof of \cref{theorem_trace_characterisation} will rely on \cref{lemma_mixed_to_boundary} and the next lemma.

\begin{lemma}
\label{lemma_uniqueness}
If \(\dimdom \in \Nset \setminus \set{0, 1}\), if \(s \in \intvo{0}{1}\), if \(p \in \intvr{1}{\infty}\), if the functions \( v, w \colon \partial \Rset^{\dimdom}_+ \to \Rset\) are measurable and if
\begin{equation}
\label{eq_soogh1zahr9aiL5eiphaeng4}
 \smashoperator[r]{\iint\limits_{\partial \Rset^{\dimdom}_+ \times \partial \Rset^{\dimdom}_+}}
 \frac{\abs{v\brk{x} - w \brk{y}}^p}{\abs{x - y}^{\dimdom - 1 + sp}} \dif y \dif x <\infty,
\end{equation}
then \(v = w\) almost everywhere.
\end{lemma}
\begin{proof}
By the assumption \eqref{eq_soogh1zahr9aiL5eiphaeng4} and Fubini's theorem, we have \(v, w \in L^p_{\mathrm{loc}} (\partial \Rset^{\dimdom}_+)\).
We define for \(\delta>0\) the functions \(v_\delta, w_{\delta} \colon \partial \Rset^{\dimdom}_+ \to \Rset\) for each  \(z \in \partial \Rset^{\dimdom}_+\) by
\begin{align*}
 v_\delta \brk{z} &\defeq \fint_{\Bset^{\dimdom - 1}_\delta \brk{z}} v&
 &\text{ and }&
 w_\delta \brk{z} &\defeq \fint_{\Bset^{\dimdom - 1}_\delta \brk{z}} w.
\end{align*}
We have for every \(z \in \Rset^{\dimdom}\) and \(\delta > 0\), by Jensen's inequality
\begin{equation}
\label{eq_aiph1Oov1ohng7Aithoor2sh}
\begin{split}
\abs{v_\delta \brk{z} - w_\delta \brk{z}}^p
&\le
\abs[\bigg]{\fint_{\Bset^{\dimdom - 1}_\delta \brk{z}} \fint_{\Bset^{\dimdom - 1}_\delta \brk{z}} \abs{v \brk{x} - w \brk{y}}\dif x\dif y }^p\\
&\le \fint_{\Bset^{\dimdom - 1}_\delta \brk{z}} \fint_{\Bset^{\dimdom - 1}_\delta \brk{z}} \abs{v \brk{x} - w \brk{y}}^p \dif x \dif y\\
&\le (2 \delta)^{\dimdom - 1 + sp} \fint_{\Bset^{\dimdom - 1}_\delta \brk{z}} \fint_{\Bset^{\dimdom - 1}_\delta \brk{z}} \frac{\abs{v \brk{x} - w \brk{y}}^p}{\abs{x - y}^{\dimdom - 1 + sp}} \dif x \dif y.
\end{split}
\end{equation}
and hence, integrating \eqref{eq_aiph1Oov1ohng7Aithoor2sh},
\begin{equation}
\label{eq_iegh7Gah6wiel6ho6yareiPh}
\begin{split}
 \int_{\partial \Rset^{\dimdom}_+}\abs{v_\delta - w_\delta}^p
 & \le \frac{2^{\dimdom - 1 + sp}}{\delta^{\dimdom - 1 - sp} \abs{\Bset^{\dimdom - 1}_1}^2}
 \int_{\partial \Rset^{\dimdom }_+} \smashoperator[r]{\iint_{\Bset^{\dimdom}_\delta \brk{z} \times \Bset^{\dimdom}_\delta \brk{z}}} \abs{v \brk{x} - w \brk{y}}^p \dif y \dif x \dif z \\
 & \le \frac{2^{\dimdom - 1 + sp}}{\delta^{\dimdom - 1 - sp} \abs{\Bset^{\dimdom - 1}_1}^2}
 \iint\limits_{\partial \Rset^{\dimdom}_+ \times \partial \Rset^{\dimdom}_+}
 \smashoperator[r]{\int_{\Bset^{\dimdom - 1}_\delta (\frac{x + y}{2})}} \frac{\abs{v \brk{x} - w \brk{y}}^p}{\abs{x - y}^{\dimdom - 1 + sp}} \dif z \dif y \dif x\\
 &\le \frac{2^{\dimdom - 1 + sp} \delta^{sp}}{\abs{\Bset^{\dimdom - 1}_1}}
 \smashoperator[r]{
 \iint_{\partial \Rset^{\dimdom}_+ \times \partial \Rset^{\dimdom}_+}} \frac{\abs{v \brk{x} - w \brk{y}}^p}{\abs{x - y}^{\dimdom - 1 + sp}} \dif y \dif x.
 \end{split}
\end{equation}
Therefore, we deduce from \eqref{eq_iegh7Gah6wiel6ho6yareiPh}  and \eqref{eq_soogh1zahr9aiL5eiphaeng4} that
\begin{equation}
 \int_{\Rset^{\dimdom}} \abs{v - w}^p \le \lim_{\delta \to 0} \int_{\Rset^{\dimdom}} \abs{v_\delta \brk{z} - w_\delta \brk{z}}^p = 0,
\end{equation}
which implies that \(v = w\) almost everywhere on \(\partial \Rset^{\dimdom}_+\).
\end{proof}
\begin{proof}[Proof of \cref{theorem_trace_characterisation}]
This follows from \cref{lemma_mixed_to_boundary} and \cref{lemma_uniqueness}.
\end{proof}

\section{Controlling the Hardy integral by the conjunction integral}

We now show how the conjunction inequality implies a Hardy inequality.
\begin{theorem}
\label{lemma_mixed_to_interior}
If \(\dimdom \in \Nset \setminus \set{0, 1}\), if \(s \in \intvo{0}{1}\), if \(p \in \intvr{1}{\infty}\), and if the functions \(u \colon \Rset^{\dimdom}_+ \to \Rset\) and \(v \colon \partial \Rset^\dimdom_+ \to \Rset\) are measurable, then
\begin{equation}
\label{eq_eV8Ohpheefuovaingoo4iong}
 \int_{\Rset^{\dimdom}_+}
 \frac{\abs{v \brk{x'} - u \brk{x}}^p}{x_{\dimdom}^{sp + 1}} \dif x
 \le
 C   \int_{\partial \Rset^{\dimdom}_+}
 \brk[\bigg]{\int_{\Rset^{\dimdom}_+}
 \frac{\abs{v \brk{x} - u \brk{y} }^p}
 {\abs{x - y}^{\dimdom + sp}}
 \dif x}
 \dif y,
\end{equation}
for some constant \(C\) depending only on \(\dimdom\), \(s\) and \(p\).
\end{theorem}
\begin{proof}
Defining the set
\[
 A
 \defeq
 \set[\big]{\brk{x, y} = \brk{x, y', y_{\dimdom}} \in \partial \Rset^{\dimdom}_+ \times \Rset^{\dimdom}_+ \st \abs{x - y'} \le 3 y_{\dimdom}/4},
\]
we have by convexity
\begin{equation}
\label{eq_quiez1jieb2pheimoola0Quo}
\begin{split}
 \iint\limits_{A} \frac{\abs{v \brk{y'} - u \brk{y}}^p}{\abs{x - y}^{\dimdom + sp}} \dif x \dif y
 &\le 2^{p - 1} \iint\limits_{A} \frac{\abs{v \brk{y'} - v \brk{x}}^p}{\abs{x - y}^{\dimdom + sp}} \dif x \dif y\\
 &\qquad + 2^{p - 1} \iint\limits_{A} \frac{\abs{v \brk{x} - u \brk{y}}^p}{\abs{x - y}^{\dimdom + sp}} \dif x \dif y.
\end{split}
\end{equation}
For the integral in the left-hand side of \eqref{eq_quiez1jieb2pheimoola0Quo}, we note that
if \(\abs{x - y'} \le y_{\dimdom}/2\), then
\[
  \abs{x - y}^2 = \abs{x - y'}^2 + y_\dimdom^2
  \le \frac{25}{16} y_{\dimdom}^2
\]
and thus
\begin{equation}
\label{eq_ZiengohS8Uxahjoo0ooc0uza}
 \smashoperator[r]{\int\limits_{\substack{x \in \partial \Rset^{\dimdom}_+\\ \abs{x - y'} \le 3 y_{\dimdom}/4}}}
  \frac{1}{\abs{x - y}^{\dimdom + sp}} \dif x
 \ge
 \smashoperator[r]{\int\limits_{\substack{x \in \partial \Rset^{\dimdom}_+\\ \abs{x - y'} \le 3 y_{\dimdom}/4}}}\;
  \frac{1}{\brk{5 y_{\dimdom}/4}^{\dimdom + sp}} \dif x = \frac{3^{\dimdom - 1} 4^{sp+1}\abs{\Bset^{\dimdom - 1}_1}}{5^{\dimdom + sp}y_{\dimdom}^{sp + 1}},
\end{equation}
whereas for the first term in the right-hand side of \eqref{eq_quiez1jieb2pheimoola0Quo}, we have
\begin{equation}
\label{eq_EiqueeleeghaeW1ooha9uoth}
\begin{split}
\int_{4\abs{x - y'}/3}^\infty
  \frac{1}{\abs{x - y}^{\dimdom + sp}} \dif x
  &\le \int_{4\abs{x - y'}/3}^\infty
  \frac{1}{y_{\dimdom}^{\dimdom + sp}} \dif y_{\dimdom}
  \\
 &= \frac{3^{\dimdom - 1 + sp}}{\brk{\dimdom - 1 + sp}4^{\dimdom - 1 + sp}\abs{x - y'}^{\dimdom -1 + sp}}.
\end{split}
\end{equation}
Inserting \eqref{eq_ZiengohS8Uxahjoo0ooc0uza} and \eqref{eq_EiqueeleeghaeW1ooha9uoth} into \eqref{eq_quiez1jieb2pheimoola0Quo}, we get
 \begin{equation}
\label{eq_foatun4aivohl7mikeiSh6ei}
\begin{split}
 &\frac{3^{\dimdom - 1} 4^{sp + 1}\abs{\Bset^{\dimdom - 1}_1}}{5^{\dimdom + sp}} \int_{\Rset^{\dimdom}_+} \frac{\abs{v \brk{y'} - u \brk{y}}^p}{y_\dimdom^{sp + 1}} \dif x \dif y\\
 &\qquad \le 2^{p - 1} \smashoperator{\iint_{\partial \Rset^{\dimdom}_+ \times \partial \Rset^{\dimdom}_+}} \frac{\abs{v \brk{x} - v \brk{y}}^p}{\abs{x - y}^{\dimdom - 1 + sp}} \dif x \dif y\\
 &\qquad \qquad +\frac{2^{p - 1} 3^{\dimdom - 1 + sp}}{\brk{\dimdom - 1 + sp}4^{\dimdom - 1 + sp}}
 \int_{\partial \Rset^{\dimdom}_+}
 \brk[\bigg]{\int_{\Rset^{\dimdom}_+}
 \frac{\abs{v \brk{x} - u \brk{y} }^p}
 {\abs{x - y}^{\dimdom + sp}}
 \dif x}
 \dif y
 .
\end{split}
\end{equation}
Thanks to \cref{theorem_mixed_to_gagliardo}, \eqref{eq_eV8Ohpheefuovaingoo4iong} follows then from \eqref{eq_foatun4aivohl7mikeiSh6ei}.
\end{proof}

\section{From Hardy and Gagliardo to conjunction}

As a converse to \cref{lemma_mixed_to_interior} and \cref{theorem_trace_characterisation}, the conjunction integral is controlled by the Gagliardo and Hardy integrals.

\begin{theorem}
\label{theorem_Hardy_Gagliardo_conjunction}
If \(\dimdom \in \Nset \setminus \set{0, 1}\), if \(s \in \intvo{0}{1}\), if \(p \in \intvr{1}{\infty}\) and if the functions \(u \colon \Rset^{\dimdom}_+ \to \Rset\) and \( v \colon \partial \Rset^{\dimdom}_+ \to \Rset\) are measurable, then\begin{equation}
\label{eq_haweijeg4EPahnaex4iey6ae}
\begin{split}
&\int_{\partial \Rset^{\dimdom}_+}
 \brk[\bigg]{\int_{\Rset^{\dimdom}_+}
 \frac{\abs{v \brk{x} - u \brk{y} }^p}
 {\abs{x - y}^{\dimdom + sp}}
 \dif x}
 \dif y\\
 &\qquad \le
 C \brk[\bigg]{
 \smashoperator[r]{\iint_{\partial \Rset^{\dimdom}_+ \times \partial \Rset^{\dimdom}_+}} \frac{\abs{v \brk{x} - v \brk{y}}^p}{\abs{x - y}^{\dimdom - 1 + sp}} \dif x \dif y
 +
 \int_{\Rset^{\dimdom}_+} \frac{\abs{v \brk{x'} - u \brk{x}}^p}{x_\dimdom^{sp + 1}} \dif x}.
\end{split}
\end{equation}
\end{theorem}
\begin{proof}
By convexity and the triangle inequality, we write
\begin{equation}
\label{eq_uo6Ciephoiquoong2wahfooF}
\begin{split}
 \int_{\partial \Rset^{\dimdom}_+}
 \brk[\bigg]{\int_{\Rset^{\dimdom}_+}
 \frac{\abs{v \brk{x} - u \brk{y} }^p}
 {\abs{x - y}^{\dimdom + sp}}
 \dif x}
 \dif y
&\le
 2^{p - 1}\int_{\partial \Rset^{\dimdom}_+}
 \brk[\bigg]{\int_{\Rset^{\dimdom}_+}
 \frac{\abs{v \brk{x} - v \brk{y'} }^p}
 {\abs{x - y}^{\dimdom + sp}}
 \dif x}
 \dif y\\
&\qquad
+  2^{p - 1}\int_{\partial \Rset^{\dimdom}_+}
 \brk[\bigg]{\int_{\Rset^{\dimdom}_+}
 \frac{\abs{v \brk{y'} - u \brk{y} }^p}
 {\abs{x - y}^{\dimdom + sp}}
 \dif x}
 \dif y.
\end{split}
\end{equation}
For the first term in the right-hand side of \eqref{eq_uo6Ciephoiquoong2wahfooF}, we compute
\begin{equation}
\label{eq_ju3Emoo7iegh8pah3ahquaen}
 \int_0^\infty \frac{1}{\brk{\abs{x - y'}^2 + y_{\dimdom}^2}^\frac{\dimdom + sp}{2}} \dif y_{\dimdom}
 = \frac{1}{\abs{x - y'}^{\dimdom - 1 + sp}}
 \int_0^\infty \frac{1}{\brk{1 + t^2}^\frac{\dimdom + sp}{2}} \dif t,
\end{equation}
where the integral on the right-hand side is finite if \(\dimdom + sp > 1\),
whereas for the second term, we have
\begin{equation}
\label{eq_shoon7Aekaeg5toodi2chee4}
\begin{split}
  \int_{\partial \Rset^{\dimdom}_+}
  \frac{1}{\abs{x - y}^{\dimdom + sp}} \dif x
  &= \int_{\partial \Rset^{\dimdom}_+}
  \frac{1}{\brk{\abs{x - y'}^2 + y_{\dimdom}^2}^{\frac{\dimdom + sp}{2}}} \dif x\\
  &= \frac{1}{y_\dimdom^{sp + 1}}
  \int_{\Rset^{\dimdom - 1}} \frac{1}{\brk{\abs{z}^2 + 1}^{\frac{\dimdom + sp}{2}}} \dif z,
\end{split}
\end{equation}
where the integral on the right-hand side is finite if \(sp + 1 > 0\).
Inserting \eqref{eq_ju3Emoo7iegh8pah3ahquaen} and \eqref{eq_shoon7Aekaeg5toodi2chee4} into \eqref{eq_uo6Ciephoiquoong2wahfooF}, we get the conclusion \eqref{eq_haweijeg4EPahnaex4iey6ae}.
\end{proof}

\section{Bourgain-Brezis-Mironescu formulae}

For smooth functions, the conjunction integral satisfies a counterpart of the Bourgain-Brezis-Mironescu formula \eqref{eq_emiaQuohtohseitaeph1ahP4}.

\begin{theorem}
\label{theorem_BBM_formula}
If \(\dimdom \in \Nset \setminus \set{0, 1}\), if \(p \in \intvr{1}{\infty}\) and if \(
 u \in C^1_c \brk{\Bar{\Rset}^{\dimdom}_+}
\),
then
\begin{equation}
\label{eq_vah4ahvaZohZoohaith9lah9}
 \lim_{s \underset{>}{\to} 0}\brk{1 - s}
 \int_{\partial \Rset^{\dimdom}_+} \brk[\bigg]{\int_{\Rset^{\dimdom}_+} \frac{\abs{u \brk{x} - u \brk{y}}^p}{\abs{x - y}^{\dimdom + sp}} \dif y} \dif x
 =
 \frac{\pi^{\frac{\dimdom - 1}{2}} \Gamma \brk{\frac{p + 1}{2}}}
{p \Gamma \brk{\frac{\dimdom + p}{2}}}
 \int_{\partial \Rset^{\dimdom}_+} \abs{\Deriv u}^p.
\end{equation}
\end{theorem}

The constant in \eqref{eq_vah4ahvaZohZoohaith9lah9} comes from the following computation.

\begin{lemma}
\label{lemma_KNP}
If \(\dimdom \in \Nset \setminus \set{0, 1}\), if \(p \in \intvr{1}{\infty}\), then
\[
 \int_{\partial \Bset^{\dimdom}_1}
  \abs{w_1}^p  \dif w
  = \frac{2 \pi^{\frac{\dimdom - 1}{2}} \Gamma\brk{\frac{p + 1}{2}}}{\Gamma\brk{\frac{\dimdom + p}{2}}}.
\]

\end{lemma}
\begin{proof}
We have
\begin{equation}
\label{eq_fahWiat0wae4boogoo1iroqu}
\begin{split}
 \int_{\partial \Bset^{\dimdom}_1}
  \abs{w_1}^p \dif w
&= \abs{\partial \Bset^{\dimdom - 1}_1}
\int_{0}^{\pi} \brk{\sin \theta}^{\dimdom - 2} \abs{\cos \theta}^p\dif \theta\\
&= \abs{\partial \Bset^{\dimdom - 1}_1} 2\int_{0}^{\pi/2} \brk{\sin \theta}^{\dimdom - 2} \brk{\cos \theta}^p\dif \theta\\
&= \frac{2 \pi^{\frac{\dimdom - 1}{2}}} {\Gamma \brk{\tfrac{\dimdom - 1}{2}}} \frac{\Gamma \brk{\frac{\dimdom - 1}{2}}\Gamma\brk{\frac{p + 1}{2}}}{\Gamma\brk{\frac{N + p}{2}}}\\
&= \frac{2 \pi^{\frac{\dimdom - 1}{2}} \Gamma\brk{\frac{p + 1}{2}}}{\Gamma\brk{\frac{\dimdom + p}{2}}},
\end{split}
\end{equation}
which proves \eqref{eq_fahWiat0wae4boogoo1iroqu}.
\end{proof}

\begin{proof}[Proof of \cref{theorem_BBM_formula}]
We consider the set
\[
 K \defeq
 \set{x \in \partial \Rset^{\dimdom}_+
 \st
 \Bset^{\dimdom}_1 \brk{x} \cap \supp u \ne \emptyset}.
\]
Since \(\supp u \subseteq \Bar{\Rset}^{\dimdom}_+\) is compact, the
\(K\) itself is also compact.
We have
\[
\begin{split}
  \int_{\Bset^{\dimdom}_{1} \brk{x} \cap \Rset^{\dimdom}_+} \frac{\abs{u \brk{x} - u \brk{y}}^p}{\abs{x - y}^{\dimdom + sp}} \dif y
   & =  \int_{\Bset^{\dimdom}_{1} \cap \Rset^{\dimdom}_+} \frac{\abs{u \brk{x} - u \brk{x + h}}^p}{\abs{h}^{\dimdom + sp}} \dif h\\
&= \int_0^1 \int_{\partial \Bset^{\dimdom}_{1} \cap \Rset^{\dimdom}_+}
 \frac{\abs{u \brk{x} -  u \brk{x + rw}}^p}{r^{1+ sp}} \dif w \dif r\\
 &= \frac{1}{1 - s} \int_0^1 \int_{\partial \Bset^{\dimdom}_{1} \cap \Rset^{\dimdom}_+}
 \frac{\abs{u \brk{x} -  u \brk{x + t^{\frac{1}{1-s}} w}}^p}{t^{1 + \frac{p}{1 - s} - p}}\dif w \dif t,
\end{split}
\]
under the change of variable \(r = t^{1/\brk{1 - s}}\)
and thus
\begin{equation}
\label{eq_aiP8sae8cieRoh7ohciebei4}
\begin{split}
 &\lim_{s \underset{<}{\to} 1} \brk{1 - s}
  \int_{\Bset^{\dimdom}_{1} \brk{x} \cap \Rset^{\dimdom}_+} \frac{\abs{u \brk{x} - u \brk{y}}^p}{\abs{x - y}^{\dimdom + sp}} \dif y\\
  &\qquad = \lim_{s \underset{<}{\to} 1} \int_0^1 \int_{\partial \Bset^{\dimdom}_{1} \cap \Rset^{\dimdom}_+}
 \frac{\abs{u \brk{x} -  u \brk{x + t^{\frac{1}{1-s}} w}}^p}{t^{1 + \frac{p}{1 - s} - p}}\dif w \dif t\\
 &\qquad = \frac{1}{\brk{1 - s}p} \int_{\partial \Bset^{\dimdom}_{1} \cap \Rset^{\dimdom}_+} \abs{\Deriv u \brk{x} \sqb{w}}^p \dif w\\
 &\qquad = \frac{1}{\brk{1 - s}p}  \frac{\pi^{\frac{\dimdom - 1}{2}} \Gamma \brk{\frac{p + 1}{2}}}
{\Gamma \brk{\frac{\dimdom + p}{2}}}
 \int_{\partial \Rset^{\dimdom}_+} \abs{\Deriv u}^p,
\end{split}
\end{equation}
since, in view of \cref{lemma_KNP},
\begin{equation*}
\begin{split}
\int_{\partial \Bset^{\dimdom}_{1} \cap \Rset^{\dimdom}_+} \abs{\Deriv u \brk{x} \sqb{w}}^p \dif w
&= \abs{\Deriv u \brk{x}}^p \int_{\partial \Bset^{\dimdom}_1  \cap \Rset^{\dimdom}_+} \abs{w_1}^p \dif w\\
&= \frac{\abs{\Deriv u \brk{x}}^p}{2} \int_{\partial \Bset^{\dimdom}_1} \abs{w_1}^p \dif w\\
&= \frac{\pi^{\frac{\dimdom - 1}{2}} \Gamma \brk{\frac{p + 1}{2}}}
{\Gamma \brk{\frac{\dimdom + p}{2}}}\abs{\Deriv u \brk{x}}^p.
\end{split}
\end{equation*}
On the other hand,
\begin{equation}
\label{eq_Thai3chu5jiyi1beeb3meeTu}
\begin{split}
 \int_{\Bset^{\dimdom}_{1}\! \brk{x} \cap \Rset^{\dimdom}_+} \frac{\abs{u \brk{x} - u \brk{y}}^p}{\abs{x - y}^{\dimdom + sp}} \dif y
 &\le  \int_{\Bset^{\dimdom}_{1}\!\brk{x} \cap \Rset^{\dimdom}_+} \frac{\norm{\Deriv u}_{L^\infty\brk{\Rset^{\dimdom}_+}}^p}{\abs{x - y}^{\dimdom + \brk{1 - s}p}} \dif y\\
 &\le \C \norm{\Deriv u}_{L^\infty\brk{\Rset^{\dimdom}_+}}^p.
\end{split}
\end{equation}
In view of \eqref{eq_aiP8sae8cieRoh7ohciebei4} and \eqref{eq_Thai3chu5jiyi1beeb3meeTu}, we have by Lebesgue’s dominated convergence
\begin{equation}
\label{eq_imoohoo6oG9ieceirohyahz0}
 \lim_{s \underset{<}{\to} 1}
 \int_{\partial \Rset^{\dimdom}_+} \brk[\bigg]{\int_{\Bset^{\dimdom}_{1}\! \brk{x} \cap \Rset^{\dimdom}_+} \frac{\abs{u \brk{x} - u \brk{y}}^p}{\abs{x - y}^{\dimdom + sp}} \dif y} \dif x
= \frac{\pi^{\frac{\dimdom - 1}{2}} \Gamma \brk{\frac{p + 1}{2}}}
{p \Gamma \brk{\frac{\dimdom + p}{2}}} \int_{\partial \Rset^{\dimdom}_+} \abs{\Deriv u}^p.
\end{equation}
Finally, we also have by Lebesgue's dominated convergence theorem,
\begin{equation}
\label{eq_Uphool2mauJ0ooCuofomail6}
 \lim_{s \underset{<}{\to} 1} \brk{1 - s}
 \int_{\partial \Rset^{\dimdom}_+} \brk[\bigg]{\int_{\Rset^{\dimdom}_+ \setminus \Bset^{\dimdom}_{1} \!\brk{x}} \frac{\abs{u \brk{x} - u \brk{y}}^p}{\abs{x - y}^{\dimdom + sp}} \dif y} \dif x
 = 0,
\end{equation}
and the conclusion \eqref{eq_vah4ahvaZohZoohaith9lah9} follows then from \eqref{eq_imoohoo6oG9ieceirohyahz0} and \eqref{eq_Uphool2mauJ0ooCuofomail6}.
\end{proof}

A suitable asymptotic control on the conjunction integral gives some integrability of the gradient.

\begin{theorem}
\label{theorem_BBM_liminf}
If \(\dimdom \in \Nset \setminus \set{0, 1}\), if \(p \in \intvo{1}{\infty}\) and if the functions \(
 u \colon \Rset^{\dimdom}_+ \to \Rset\) and \(v \colon \Rset^{\dimdom}_+ \to \Rset\) are measurable and satisfy
\begin{equation}
\label{eq_usohVeh9siqui5Zahahtae4X}
  \liminf_{s \underset{<}{\to} 1}\; \brk{1 - s}
 \int_{\partial \Rset^{\dimdom}_+} \brk[\bigg]{\int_{\Rset^{\dimdom}_+} \frac{\abs{v \brk{x} - u \brk{y}}^p}{\abs{x - y}^{\dimdom + sp}} \dif y} \dif x
 < \infty,
\end{equation}
then \(v \in \sobolev^{1, p}\brk{\partial \Rset^{\dimdom}_+}\)
 and
\[
   \int_{\partial \Rset^{\dimdom}_+} \abs{\Deriv v}^p
   \le
   C
   \liminf_{s \underset{<}{\to} 1}\; \brk{1 - s}
 \int_{\partial \Rset^{\dimdom}_+} \brk[\bigg]{\int_{\Rset^{\dimdom}_+} \frac{\abs{v \brk{x} - u \brk{y}}^p}{\abs{x - y}^{\dimdom + sp}} \dif y} \dif x,
\]
for some constant \(C\) depending only on \(N\) and \(p\).
\end{theorem}
\begin{proof}
By \cref{theorem_mixed_to_gagliardo}, we have
\[
\begin{split}
&\liminf_{s \underset{<}{\to} 1}\; \brk{1 - s}
  \smashoperator{\iint\limits_{\partial \Rset^{\dimdom}_+ \times \partial \Rset^{\dimdom}_+}}
 \frac{\abs{v \brk{x} - v (y)}^p}{\abs{x - y}^{\dimdom - 1 + sp}} \dif x \dif y \\
 &\qquad \le  \frac{2 \cdot 3^{2 \dimdom -1 + p}\Gamma \brk{\frac{\dimdom + 2}{2}}}{4^{p} \pi^\frac{1}{2} \Gamma \brk{\frac{\dimdom + 1}{2}}}
 \liminf_{s \underset{<}{\to} 1}\brk{1 - s}
 \int_{\partial \Rset^{\dimdom}_+} \brk[\bigg]{\int_{\Rset^{\dimdom}_+} \frac{\abs{v \brk{x} - u \brk{y}}^p}{\abs{x - y}^{\dimdom + sp}} \dif y} \dif x < \infty;
\end{split}
\]
the conclusion then follows from the classical corresponding Bourgain-Brezis-Mironescu result \cite{Bourgain_Brezis_Mironescu_2001}.
\end{proof}

\begin{theorem}
\label{theorem_BBM_liminf_BV}
If \(\dimdom \in \Nset \setminus \set{0, 1}\) and if the functions \(
 u \colon \Rset^{\dimdom}_+ \to \Rset\) and \(v \colon \Rset^{\dimdom}_+ \to \Rset\) are measurable and satisfy
\begin{equation*}
  \liminf_{s \underset{<}{\to} 1}\;\brk{1 - s}
 \int_{\partial \Rset^{\dimdom}_+} \brk[\bigg]{\int_{\Rset^{\dimdom}_+} \frac{\abs{v \brk{x} - u \brk{y}}}{\abs{x - y}^{\dimdom + s}} \dif y} \dif x
 < \infty,
\end{equation*}
then   \(v \in BV \brk{\partial \Rset^{\dimdom}_+}\)
 and
\[
   \int_{\partial \Rset^{\dimdom}_+} \abs{\Deriv v}
   \le
   C \liminf_{s \underset{<}{\to} 1}\brk{1 - s}
 \int_{\partial \Rset^{\dimdom}_+} \brk[\bigg]{\int_{\Rset^{\dimdom}_+} \frac{\abs{v \brk{x} - u \brk{y}}}{\abs{x - y}^{\dimdom + s}} \dif y} \dif x,
\]
for some constant \(C\) depending only on \(N\).
\end{theorem}
\begin{proof}
This follows from \cref{theorem_mixed_to_gagliardo}
and the corresponding results for the Gagliardo integral \citelist{\cite{Davila_2002}\cite{VanSchaftingen_Willem_2004}}.
\end{proof}

\end{document}